\documentclass{amsart}
\usepackage{latexsym}
\usepackage{amssymb}
\usepackage{epsfig}
\usepackage{amsthm}
\usepackage{latexsym}
\usepackage{longtable}
\usepackage{epsfig}
\usepackage{amsmath}
\usepackage{hhline}
\usepackage{color}
\newtheorem{theorem}{Theorem}[section]

\newtheorem{lm}[theorem]{Lemma}
\newtheorem{tr}[theorem]{Theorem}

\newtheorem{rem}[theorem]{Remark}
\newtheorem{pr}[theorem]{Proposition}

\newtheorem{ex}[theorem]{Example}

\DeclareMathOperator{\GL}{GL}

\DeclareMathOperator{\Dist}{Dist}

\begin{document}
\title[A combinatorial approach to Donkin-Koppinen filtrations]{A combinatorial approach to Donkin-Koppinen filtrations of general linear supergroups}

\author{F.Marko}
\address{The Pennsylvania State University, 76 University Drive, Hazleton, 18202 PA, USA}
\email{fxm13@psu.edu}

\begin{abstract}
For a general linear supergroup $G=\GL(m|n)$, we consider a natural isomorphism $\phi: G \to U^-\times G_{ev} \times U^+$, where $G_{ev}$ is the even subsupergroup of $G$, and $U^-$, $U^+$ are appropriate odd unipotent subsupergroups of $G$. We compute the action of odd superderivations on the images $\phi^*(x_{ij})$ of the generators of $K[G]$, 
extending results established in \cite{prim} and \cite{irred}.

We describe a specific ordering of the dominant weights $X(T)^+$ of $GL(m|n)$ for which there exists a Donkin-Koppinen filtration of the coordinate algebra $K[G]$.
Let $\Gamma$ be a finitely generated ideal $\Gamma$ of $X(T)^+$ and
$O_{\Gamma}(K[G])$ be the largest $\Gamma$-subsupermodule of $K[G]$ having simple composition factors of highest weights $\lambda\in \Gamma$.
We apply combinatorial techniques, using generalized bideterminants, to determine a basis of $G$-superbimodules appearing in Donkin-Koppinen filtration of $O_{\Gamma}(K[G])$, considered initially in \cite{markozub3}.
\end{abstract}

\maketitle

\section*{Introduction and Donkin-Koppinen filtrations}

\subsection*{Algebraic group $G$}

Let $G$ be a (connected) reductive affine algebraic group over an algebraically closed field K, and $K[G]$ be its coordinate algebra. 
Let $\Phi^+$ be the set of positive roots of $G$, and $X(T)^+$ be the corresponding set of dominant weights of $G$.
The set $X(T)^+$ is considered with respect to the Bruhat-Tits dominance order $\unlhd$ such that  $\mu\unlhd \lambda$ if and only if $\lambda-\mu$ is a sum of positive roots from $\Phi$. Denote by $w_0$ the longest element of the Weyl group of $G$, 
and for $\lambda\in X(T)^+$, denote $\lambda^*=-w_0 \lambda$. 
Denote by $H^0_G(\lambda)$ the induced module of the highest weight $\lambda$. Then 
$H^0_G(\lambda^*)$ is isomorphic to the dual of the Weyl module $V_G(\lambda)$ of the highest weight $\lambda$.

Regard $K[G]$ as a left $G\times G$-module via the action 
\[(g_1,g_2)(f)(g)=f(g_2^{-1}gg_1)\]
for $g_1,g_2,g\in G$ and $f\in K[G]$.

Let $\Gamma$ be a finite ideal of the set $X(T)^+$ dominant weights of $G$. 
For a (left) $G$-module $V$, denote by $O_{\Gamma}(V)$ the maximal submodule of $V$ such that all its composition factors are of irreducible modules $L(\lambda)$ with $\lambda\in \Gamma$.

List elements $\lambda_1, \lambda_2, \ldots, \lambda_i, \ldots$ of $X(T)^+$ is such a way that $\lambda_i\unlhd \lambda_j$ implies $i\leq j$, and denote the ideal $\{\lambda_1, \ldots, \lambda_m\}$ of $X(T)^+$ by $\pi_m$.

Donkin (see p. 472 of \cite{don}) proved that $K[G]$ has a good filtration by $G\times G$-submodules $V_m=O_{\pi_m}(K[G])$
such that the quotients 
$V_m/V_{m-1} \simeq H^0_G(\lambda_m) \otimes H^0_G(\lambda_m^*)$ for each $m\geq 1$.
An analogous result about the filtration of $O_{\Gamma}(K[G])$ for every ideal $\Gamma$ of $X(T)^+$ is also valid.
We will call the above filtration a {\it Donkin-Koppinen filtration}. 

Koppinen in \cite{kop} proved the existence of such filtrations using a slightly different notation that will be more suitable for this paper.

If $V$ is a left $G$-module, then denote by $V^r$ the corresponding right $G$-module, where the right action of $G$ on the $K$-space $V$ is given by 
$v.g= g^{-1}.v$ for $g\in G$ and $v\in V$.
There is a bijective correspondence between the (left) $G\times G$-module structure on a $K$-space $V$ and the (left-right) $G$-bimodule structure on $V$ given by
$(g_1,g_2)v=g_1vg_2^{-1}$ for $g_1,g_2\in G$ and $v \in V$.
Following \cite{kop}, if $V=V_1\otimes V_2$ and $V_1, V_2$ are (left) $G$-modules, we write $V_1\otimes V_2^r$ for the corresponding (left-right) $G$-bimodule. 
Thus, the terms in the good filtration of the $G$-bimodule $K[G]$ are $G$-bimodules and the consecutive quotients in this filtration are isomorphic to $H^0_G(\lambda_m)\otimes H^0_G(\lambda_m^*)^r$ as $G$-bimodules.

Both Donkin and Koppinen used methods of homological algebra.

\subsection*{Schur algebras and general linear groups}

In the case when $G=GL(m)$ is the general linear group, the polynomial representations correspond to modules over the Schur algebra $S=S(m,r)$ of degree $r\geq 0$.
The Schur algebra $S$ is a dual of the bialgebra of polynomials $A(m,r)$ of degree $r$.
Denote by $M_S(\lambda)$ the left Schur $S$-module of the highest weight $\lambda$, which is isomorphic to the left costandard module $\nabla_S(\lambda)$ of the highest weight $\lambda$.
Also, denote by $V_S(\lambda)$ the left Weyl $S$-module of the highest weight $\lambda$, which is isomorphic to the
left standard module $\Delta_S(\lambda)$ of the highest weight $\lambda$.

It is a standard result (see \cite{martin}) that $A(m,r)$ has a filtration by (left and right) $S$-bimodules 
$A_{\leq \lambda}$ that are given as a $K$-span of bideterminants of shape $\mu \unlhd \lambda$. 
The factorbimodules $A_{\leq \lambda}/A_{<\lambda}$ are isomorphic to $M_S(\lambda)\otimes V_S(\lambda)^*\simeq \nabla_S(\lambda)\otimes \Delta_S(\lambda)^*$ as $S$-bimodules, 
where $*$ denotes the $Hom(-,K)$ dual. 
Here $V_S(\lambda)^*$ is isomorphic to the right Schur (costandard) $S$-module of the highest weight $\lambda$.

The algebra $S$ is quasi-hereditary, and it has a filtration by heredity chains with quotients 
$M_S(\lambda)^*\otimes V_S(\lambda) \simeq \nabla_S(\lambda)^*\otimes \Delta_S(\lambda)$. This heredity chain filtration of $S(m,r)$ is dual to the above-described filtration of $A(m,r)$.

If $\mu $ is a dominant weight of $G=GL(m)$, then up to a tensoring by a suitable power of determinant $D=Det$ of $G$, the left induced $\GL(m)$-module $H^0_G(\mu)$ is obtained from a Schur (costandard) module $M_S(\lambda)$ of a corresponding polynomial weight $\lambda$. 
Again, up to a shift by the determinant $D$, the left induced $\GL(m)$-module $H^0_G(\mu^*)$ is isomorphic to $H^0_G(\lambda^*)$ for a corresponding polynomial weight $\lambda$. The module $H^0_G(\lambda^*)$ is isomorphic to the dual $V_G(\lambda)^*$ of the Weyl module $V_G(\lambda)$  (see p.182 of \cite{jan} and p.67 of \cite{martin}). 
Note that $V_G(\lambda)^*$ is a left $G$-module, and $(V_G(\lambda)^*)^r \simeq V_S(\lambda)^*$ is a right $S$-module. 
Thus,
the $G$-bimodule $H^0_G(\mu)\otimes H^0_G(\mu^*)^r$ is isomorphic, up to a shift
by a power of $D$, to the $S$-bimodule $M_S(\lambda)\otimes V_S(\lambda)^*$, which is described combinatorially using bideterminants.  

\subsection*{General linear supergroup $GL(m|n)$}
Let $G=GL(m|n)$ be a general linear supergroup, and $X(T)^-$ be a set of dominant weights of $G$ corresponding to the set of negative roots $\Phi^-$ of $G$.

Let $\Lambda$ be a finitely-generated ideal in $X^-\times X(T)^+$, and $O_{\Lambda}(K[G])$ be the largest $G\times G$-subsupermodule of $K[G]$ such that all of its composition factors are simple $G\times G$-supermodules of a highest weight $\lambda\in \Lambda$.

Theorem 6.1 of \cite{sz} states that for every finitely generated ideal $\Lambda\subseteq X^-\times X(T)^+$, the $G\times G$-supermodule $O_{\Lambda}(K[G])$ has a decreasing good filtration
\[O_{\Lambda}(K[G])=V_0\supseteq V_1\supseteq V_2\supseteq\ldots\]
such that 
\[V_k/V_{k+1}\simeq V_G(\lambda_k)^*\otimes H^0_G(\lambda_k)\]
as $G\times G$-supermodules for $k\geq 0$. 

By Proposition 5.2 of \cite{markozub3}, if $\Lambda= (-\Gamma)\times \Gamma$, where $ \Gamma$ is a finitely generated ideal of $X(T)^+$ of $G$, then $O_{\Lambda}(K[G])=O_{\Gamma}(K[G])$.

Let $G_{ev}\simeq GL(m)\times GL(n)$ be the even subgroup of $GL(m|n)$.
For an ideal $\Gamma$ of $X(T)^+$, denote by $M_{\Gamma}$ the $G_{ev}$-module $O_{\Gamma}(K[G_{ev}])$, 
and by $C_{\Gamma}$ the $G$-supermodule $O_{\Gamma}(K[G])$. 

If $\lambda$ is a maximal element of $\Gamma$, then the factormodule $M_{\Gamma}/M_{\Gamma\setminus\{\lambda\}}$ has a $K$-basis 
consisting of generalized bideterminants in the terminology of \cite{markozub3}. Using this, 
an explicit basis of the factorsupermodules $C_{\Gamma}/C_{\Gamma\setminus\{\lambda\}}$ was given in Theorem 7.3 of \cite{markozub3}.

\subsection*{Goals of the paper}

It is natural to ask if the $G=GL(m|n)$-supermodules $C_{\Gamma}$ can be described purely combinatorially using the actions of the left and right superderivations. 
The affirmative answer to this question is the primary goal of this paper. To accomplish this, 
we consider a natural isomorphism $\phi: G \to U^-\times G_{ev} \times U^+$, where $G_{ev}$ is the even subsupergroup of $G$, and $U^-$, $U^+$ are appropriate odd unipotent subsupergroups of $G$. 
We describe an explicit ordering of the dominant weights of $GL(m|n)$ and completely describe the action of superderivations on the images $\phi^*(x_{ij})$ of the generators of $K[G]$, 
which is of independent interest and is an extension of the results from \cite{irred} and \cite{prim}.
That way, we obtain a hands-on description of the Donkin-Koppinen filtration for $O_{\Gamma}(K[G])$ for $G=\GL(m|n)$ considered initially in \cite{markozub3}.

The structure of the paper is as follows. 

In Section 1, we introduce additional notation specialized for $G=GL(m|n)$. 
In Section 2, we explicitly describe Donkin-Koppinen filtrations for $G=GL(1|1)$.
In Section 3, we determine a basis of the modules $M_{\Gamma}=O_{\Gamma}(K[G_{ev}])$ for $GL(m|n)$.
In Section 4, we determine the action of superderivations on element $\phi^*(x_{ij})$, find a basis of $C_{\Gamma}=O_{\Gamma}(K[G])$ for $GL(m|n)$ and Donkin-Koppinen filtration of $O_{\Gamma}(K[G])$.

\section{Additional notation}

For general notation and basic setup involving supergroups and their representations, standard and costandard supermodules, please consult \cite{brunkuj, zub1, zub3}.

Let $X$ denote the \emph{generic matrix} $(x_{ij})_{1\leq i, j\leq m+n}$, represented by blocks as
\[X=\left(\begin{array}{cc}
X_{11} & X_{12} \\
X_{21} & X_{22}
\end{array}\right),\]
where
$X_{11}=(x_{ij})_{1\leq i, j\leq m}$, $X_{12}=(x_{ij})_{1\leq i\leq m < j\leq m+n}$,
$X_{21}=(x_{ij})_{1\leq j\leq m < i\leq m+n}$ and $X_{22}=(x_{ij})_{m< i, j\leq m+n}$.
The variables $x_{ij}$, where $1\leq i,j\leq m$ or $m+1\leq i,j \leq m+n$ are designated as even ($|x_{ij}=0|$), 
and the remaining variables $x_{ij}$ are designated as odd ($|x_{ij}|=1$.)
We denote $D_1=det(X_{11})$ and $D_2=det(X_{22})$.

Denote by $K[GL(m|n)]$ the Hopf superalgebra
\[K[x_{ij}\mid 1\leq i, j\leq m+n]_{D_1D_2},\]
such that its comultiplication is defined by
\[\Delta_{\GL(m|n)}(x_{ij})=\sum_{1\leq k\leq m+n} x_{ik}\otimes x_{kj}\]
and its counit is defined as
\[ \epsilon_{\GL(m|n)}(x_{ij})=\delta_{ij}.\]

The \emph{general linear supergroup} $\GL(m|n)$ is a functor from the category of commutative superalgebras $SAlg_K$ 
to the category of groups such that 
\[\GL(m|n)(A) = Hom_{SAlg_K}(K[\GL(m|n)], A)\] for $A\in SAlg_K$.
From now, we write $G$ in place of $GL(m|n)$.

Denote $\ker\epsilon_G=K[G]^+$ by $\mathfrak{m}$. For every integer $t\geq 0$, denote the space
$(K[G]/\mathfrak{m}^{t+1})^*$ by $\Dist_t(G)$.
The \emph{distribution superalgebra} $\Dist(G)$ of $G$ is defined as $\Dist(G)=\cup_{t\geq 0} \Dist_t(G)\subseteq K[G]^*$. 

The superspace $\Dist_1(G)^+ =(\mathfrak{m}/\mathfrak{m}^2)^*$ is identified with the Lie superalgebra 
$\mathfrak{g}=\mathfrak{gl}(m|n)$ of $G=GL(m|n)$. Denote by $e_{ij}$ the basis element of $\mathfrak{g}$ corresponding to $x_{ij}$.  

The superalgebra $K[G]$ has a natural structure of a left $G$-supermodule via the left representation $\rho_l$ of $G$. 
Another $G$-supermodule structure on $G$ is given by the right regular representation $\rho_r$ of $G$. 

The left (and right, respectively) action of $e_{ij}$ on the supermodule $K[G]$ is given by the right (and left, respectively) superderivation $_{ij}D$ (and $D_{ij}$, respectively). These actions are determined by 
$(x_{kl}){_{ij}D}=\delta_{ji}x_{kj}$ and $D_{ij}(x_{kl})=\delta_{jk} x_{il}$.

We consider $K[G]$ as a left $G$-supermodule via $\rho_l$, and as a $G\times G$-supermodule via $\rho_l\times \rho_r$.

Let $G_{ev}$ be the maximal even subsupersubgroup of $G$, and $U^-$ and $U^+$ be the purely odd unipotent subsupergoups of $G$ such that $\phi: G\to U^-\times G_{ev} \times U^+$ is the isomorphism from Section 6 of \cite{markozub3} given by 
\[\left(\begin{array}{cc}
A_{11} & A_{12} \\
A_{21} & A_{22}
\end{array}\right)\stackrel{\phi}{\mapsto} \]
\[\left(\begin{array}{cc}
I_m & 0 \\
A_{21}A_{11}^{-1} & I_m
\end{array}\right)\times \left(\begin{array}{cc}
A_{11} & 0 \\
0 & A_{22}-A_{21}A_{11}^{-1}A_{12}
\end{array}\right)\times \left(\begin{array}{cc}
I_m &A_{11}^{-1} A_{12} \\
0 & I_n
\end{array}\right).\]
The dual isomorphism of coordinate superalgebras 
\[\phi^* : K[U^-]\otimes K[G_{ev}]\otimes K[U^+]\to K[G]\]
is defined by
\[X_{21}\mapsto X_{21}X_{11}^{-1}, \quad X_{11}\mapsto X_{11}, \quad X_{22}\mapsto X_{22}-X_{21}X_{11}^{-1}X_{12}, \text{ and } 
X_{12}\mapsto X_{11}^{-1}X_{12}.\]
In what follows, we write $y_{ij}=\phi^*(x_{ij})$.

Every ideal $\Gamma$ of $X(T)^+$ decomposes as a direct sum $\Gamma=\oplus_{r\in \mathbb{Z}} \Gamma_r$, 
where $\Gamma_r$ consists of those elements $\lambda\in \Gamma$ that have length $|\lambda|=r$.
Corresponding to this decomposition of $\Gamma$, we obtain the decomposition of 
$M_{\Gamma}=O_{\Gamma}(K[G_{ev}])=\oplus_{r\in \mathbb{Z}} O_{\Gamma_r}(K[G_{ev}])$ and 
$C_{\Gamma}=O_{\Gamma}(K[G])=\oplus_{r\in \mathbb{Z}} O_{\Gamma_r}(K[G])$.
Therefore, from now on, we will assume that $\Gamma=\Gamma_r$ for some integer $r$.

\section{Donkin-Koppinen filtrations for $G=\GL(1|1)$}

Before we describe a basis of the supermodules in Donkin-Koppinen filtrations for $\GL(m|n)$ in general, we discuss the simplest example when $m=n=1$.

Every finitely generated ideal $\Gamma=\Gamma_r$ is principal, and generated by $\lambda=(k, r-k)$ for some integer $k$. We will explicitly describe the filtration for $C_{\Gamma}=C_{(\lambda]}$.

From the definition of the morphism $\phi^*$ we get 
\[y_{11}=\phi^*(x_{11})=x_{11}, y_{12}=\phi^*(x_{12})=x_{11}^{-1}x_{12}, y_{21}=\phi^*(x_{21})=x_{11}^{-1}x_{12}\]
and 
\[y_{22}=\phi^*(x_{22})=x_{22}-x_{21}x_{11}^{-1}x_{12}.\]

The next lemma describes the action of superderivations on elements $y_{ij}$ as above.

\begin{lm} We have the following formulae.
\[\begin{aligned}&(y_{11})_{12}D=y_{11}y_{12}, &(y_{11})_{21}D=0, & \quad D_{21}(y_{11})=y_{21}y_{11}, & D_{12}(y_{11})=0;\\
&(y_{12})_{12}D=0,&(y_{12})_{21}D=1,& \quad D_{21}(y_{12})=y_{11}^{-1}y_{22},& D_{12}(y_{12})=0;\\
&(y_{21})_{12}D=y_{11}^{-1}y_{22},&(y_{21})_{21}D=0,&\quad D_{21}(y_{21})=0,& D_{12}(y_{21})=1;\\
&(y_{22})_{12}D=y_{22}y_{12},&(y_{22})_{21}D=0, &\quad D_{21}(y_{22})=y_{21}y_{22}, & D_{12}(y_{22})=0.
\end{aligned}\] 
\end{lm}
\begin{proof}
Straightforward from the definitions.
\end{proof}

Let $\pi=(1|-1)$, and fix $\mu=(i|j)$ such that $i+j=r$. 
Denote 
\[A_{\mu}=y_{11}^iy_{22}^j, B_{\mu}=y_{11}^iy_{22}^jy_{12}, C_{\mu}=y_{11}^{i}y_{22}^{j}y_{21}, D_{\mu}=y_{11}^{i}y_{22}^{j}y_{12}y_{21}.\]

Using the above lemma, we obtain the following. 

\begin{lm}\label{l2}
\[\begin{aligned}
&(A_{\mu})_{12}D\!\!&=&rB_{\mu}, &(A_{\mu})_{21}D&=\!\!&0, \\
&\quad D_{21} (A_{\mu})\!\!&=&rC_{\mu}, &D_{12}(A_{\mu})&=\!\!&0;\\
&(B_{\mu})_{12}D\!\!&=&0, &(B_{\mu})_{21}D&=\!\!&A_{\mu},  \\
&\quad D_{21} (B_{\mu})\!\!&=&A_{\mu-\pi}-rD_{\mu}, &D_{12}(B_{\mu})&=\!\!&0;\\
&(C_{\mu})_{12}D\!\!&=&A_{\mu-\pi}-rD_{\mu}, &(C_{\mu})_{21}D&=\!\!&0, \\
&\quad D_{21} (C_{\mu})\!\!&=&0, &D_{12}(C_{\mu})&=\!\!&A_{\mu};\\
&(D_{\mu})_{12}D\!\!&=&B_{\mu-\pi}, &(D_{\mu})_{21}D&=\!\!&-C_{\mu}, \\
&\quad D_{21}(D_{\mu})\!\!&=&C_{\mu-\pi},
&D_{12}(D_{\mu})&=\!\!&-B_{\mu}.
\end{aligned}\]
\end{lm}
\begin{proof}
The proof is left for the reader.
\end{proof}

\begin{pr}
Let $\lambda=(k,r-k)$, $\Gamma=(\lambda]$ and $\Lambda=(-\Gamma)\times \Gamma$. Then the superbimodules $V_i $ of the Donkin-Koppinen filtration of $O_{\Gamma}(K[G])$ are given as 
the $K$-span of monomials $A_{\lambda-u\pi}$, $B_{\lambda-u\pi}$,
$C_{\lambda-u\pi}$ and $D_{\lambda-u\pi}$ for $i\leq u$.
\end{pr}
\begin{proof}
We will apply formulae from Lemma \ref{l2}.

If $p$ does not divide $r$, then  $_{\mu}W=KA_{\mu}+KB_{\mu}$ and 
$_{\mu-\pi}V=KC_{\mu}+K(A_{\mu-\pi}-rD_{\mu})$ are left simple $G$-supermodules $L_G(\mu)$ and $L_G(\mu-\pi)$ of highest weight $\mu$ and $\mu-\pi$, respectively. The $K$-spaces $W_{\mu}=KC_{\mu}+KA_{\mu}$ and 
$V_{\mu-\pi}=KB_{\mu}+K(A_{\mu-\pi}-rD_{\mu})$ are right simple $G$-supermodules $L_G(\mu)^r$ and $L_G(\mu-\pi)^r$ of highest weight $\mu$ and $\mu-\pi$, respectively. Additionally, $KA_{\mu}+KB_{\mu}+KC_{\mu}+K(A_{\mu-\pi}-rD_{\mu})$ 
is a $G$-superbimodule, generated by $A_{\mu}$, that is isomorphic to $L_G(\mu)\oplus L_G(\mu-\pi)$ as left $G$-supermodule,  isomorphic 
to $L_G(\mu)^r\oplus L_G(\mu-\pi)^r$ as right $G$-supermodule, and isomorphic to $ H^0_G(\mu)\otimes V_G(\mu)^*=
L_G(\mu)\otimes L_G(\mu)$ as a $G\times G$-supermodule.

It follows that the space $V_i$ is a $G\times G$-supermodule for every $0\leq i$, and that 
$V_i/V_{i+1}\simeq H^0_G(\lambda-i\pi)\otimes V_G(\lambda-i\pi)^*$ as $G\times G$-superbimodules.

If $p$ divides $r$, then the left projective-injective supermodule $I_G(\mu)$
has the socle $L_G(\mu)$ generated by $A_{\mu}=A_{\mu}-rD_{\mu+\pi}$ and the top $L_G(\mu)$ generated by $D_{\mu+\pi}$. In this case, $_{\mu}W$ is the costandard supersubmodule $H^0_G(\mu)$ of $I_G(\mu)$, the 
factorsupermodule $I_G(\mu)/_{\mu}W$, represented by linear combinations of classes of elements $\overline{D_{\mu+\pi}}$ and $\overline{C_{\mu+\pi}}$, is isomorphic to the costandard supermodule  $H^0_G(\mu+\pi)$, and $I_G(\mu)$ has the 
left regular representation
\[\begin{array}{ccccccc}&&&L_G(\mu)&&&\\
&&\diagup&&\diagdown&\\
&L_G(\mu-\pi)&&&&&L_G(\mu+\pi)\\
&&\diagdown&&\diagup&\\
&&&L_G(\mu)&&&
\end{array}.\]

The space $V_i$ is a $G$-superbimodule for every $0\leq i$. The quotient $V_i/V_{i+1}$, generated by $D_{\mu-i\pi}$ as a $G$-superbimodule, is a $K$-span of 
elements $A_{\mu-i\pi}$, $B_{\mu-i\pi}$ $C_{\mu-i\pi}$ and $D_{\mu-i\pi}$. It is isomorphic to 
$H^0_G(\lambda-i\pi)\otimes V_G(\lambda-i\pi)^*$ as a $G\times G$-supermodule.
\end{proof}

\section{A basis of $M_{\Gamma}$ for $\GL(m|n)$}

In this section, we use the combinatorics of tableaux and bideterminants to explicitly construct a basis of bimodules $M_{\Gamma}$ for a finitely generated ideal $\Gamma$ of $X(T)^+$. 

First, we consider that case when $G=\GL(m)$.

\subsection{$\GL(m)$-modules based on bideterminants}

We start by recalling a few classical facts about tableaux, bideterminants, and Schur algebras. The reader is asked to consult \cite{martin} for more details and explanations.
Denote by $A(m)$ the bialgebra that is a polynomial algebra freely generated in commuting variables $c_{ij}$ for $1\leq i,j\leq m$, with the comultiplication $\Delta(c_{ij})=\sum_{k=1}^m c_{ik}\otimes c_{kj}$ and the counit
$\epsilon(c_{ij})=\delta_{ij}$. The dual $A(m)^*$ is an algebra, and its degree $r\geq 0$ component 
$A(m,r)^*=S(m,r)$ is the Schur algebra of degree $r$.

We will denote partitions of $\tilde{r}$ consisting of no more than $m$ parts by symbols $\tilde{\lambda}$, to distinguish it from dominant weights $\lambda$ of the group $GL(m)$. Then  
 $\tilde{r}=|\tilde{\lambda}|=\sum_{i=1}^m \tilde{\lambda}_i$ is the length of the partition $\tilde{\lambda}$.

Partitions $\tilde{\lambda}$ of $\tilde{r}$ are ordered by reverse lexicographic order on their conjugates. This means $\tilde{\lambda}\geq \tilde{\mu}$ if and only if $\tilde{\lambda}'\leq_{lex} \tilde{\mu}'$, where $'$ denotes the conjugate partition.

\begin{lm}\label{l12.1}
The order $\leq$ refines the dominance order $\unlhd$. That is,  
$\tilde{\lambda}\unlhd \tilde{\mu}$ implies $\tilde{\lambda}\leq \tilde{\mu}$.
\end{lm}
\begin{proof}
It is well known (see p. 210 of \cite{cl2}) that the lexicographic order refines the dominance order. 
Also, for partitions, $\tilde{\lambda}\unlhd \tilde{\mu}$ if and only if $\tilde{\mu}'\unlhd \tilde{\lambda}'$.
Therefore, $\tilde{\lambda}\unlhd \tilde{\mu}$ implies $\tilde{\mu}'\unlhd \tilde{\lambda}'$, which yields $\tilde{\mu}'\leq_{lex} \tilde{\lambda}'$ and 
$\tilde{\lambda}\leq \tilde{\mu}$.
\end{proof}

Let us list all partitions $\tilde{\lambda}$ of $\tilde{r}$ with respect to the order $\leq$ as
\begin{equation}\label{1}\tilde{\lambda}_1=([\frac{\tilde{r}}{m}]+1, \ldots, [\frac{\tilde{r}}{m}]+1,[\frac{\tilde{r}}{m}], \ldots,  [\frac{\tilde{r}}{m}])<\ldots <\tilde{\lambda}_{t_{\tilde{r}-1}} <\tilde{\lambda}_{t_{\tilde{r}}}=(\tilde{r}).\end{equation}

The following algorithm describes how to write all partitions on the list $(\ref{1})$:
We represent a partition $\tilde{\lambda}$ by a diagram $[\tilde{\lambda}]$ of the shape $\tilde{\lambda}$. The diagram of $[\tilde{\lambda}_k]$ is obtained from $[\tilde{\lambda}_{k-1}]$ by moving the rightmost box in the last row to the end of the closest row above it so that the resulting diagram is a diagram of a partition. 

\begin{ex} Take $m=3$ and $\tilde{r}=6$.
The listing $(\ref{1})$ and the corresponding diagrams are depicted below.

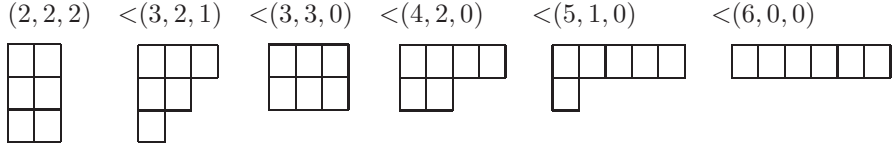
\[\begin{aligned}
&(2,2,2)&< &(3,2,1)&< &(3,3,0)&< &(4,2,0)&< &(5,1,0)&< &(6,0,0)&\\
&
\begin{tabular}{|l|l|}
\hline
 &  \\ \hline
 &  \\ \hline
 &  \\ \hline
\end{tabular}
&&
\begin{tabular}{|l|ll}
\hline
 & \multicolumn{1}{l|}{} & \multicolumn{1}{l|}{} \\ \hline
 & \multicolumn{1}{l|}{} &                       \\ \cline{1-2}
 &                       &                       \\ \cline{1-1}
\end{tabular}
&&
\begin{tabular}{lll}
\hline
\multicolumn{1}{|l|}{} & \multicolumn{1}{l|}{} & \multicolumn{1}{l|}{} \\ \hline
\multicolumn{1}{|l|}{} & \multicolumn{1}{l|}{} & \multicolumn{1}{l|}{} \\ \hline
                       &                       &                      
\end{tabular}
&&
\begin{tabular}{llll}
\hline
\multicolumn{1}{|l|}{} & \multicolumn{1}{l|}{} & \multicolumn{1}{l|}{} & \multicolumn{1}{l|}{} \\ \hline
\multicolumn{1}{|l|}{} & \multicolumn{1}{l|}{} &                       &                       \\ \cline{1-2}
                       &                       &                       &                      
\end{tabular}
&&
\begin{tabular}{lllll}
\hline
\multicolumn{1}{|l|}{} & \multicolumn{1}{l|}{} & \multicolumn{1}{l|}{} & \multicolumn{1}{l|}{} & \multicolumn{1}{l|}{} \\ \hline
\multicolumn{1}{|l|}{} &                       &                       &                       &                       \\ \cline{1-1}
                       &                       &                       &                       &                      
\end{tabular}
&&
\begin{tabular}{llllll}
\hline
\multicolumn{1}{|l|}{} & \multicolumn{1}{l|}{} & \multicolumn{1}{l|}{} & \multicolumn{1}{l|}{} & \multicolumn{1}{l|}{} & \multicolumn{1}{l|}{} \\ \hline
                       &                       &                       &                       &                       &                       \\
                       &                       &                       &                       &                       &                      
\end{tabular}
&
\end{aligned}\]
\end{ex}

Denote by $T^{\tilde{\lambda}}$ a fixed basic tableau of shape $\tilde{\lambda}$ and by $T^{\tilde{\lambda}}_{\ell_{\tilde{\lambda}}}$ the canonical tableau of shape $\tilde{\lambda}$.
Denote by $A(\leq \tilde{\lambda})$ the $K$-span of bideterminants $T^{\tilde{\zeta}}(i:j)$ for $\tilde{\zeta}\leq \tilde{\lambda}$, where $i, j$ are multi-indices of length $\tilde{r}$ with components from the set $\{1, \ldots, \tilde{r}\}$.

A classical statement (see \cite{martin}) asserts that $A(m,\tilde{r})$ has a filtration by $S(m,\tilde{r})$-bimodules 
$A(\leq \tilde{\lambda}_k)$, where $\tilde{\lambda}_k$ is listed in (\ref{1}), such that the quotient 
$A(\leq \tilde{\lambda}_k)/A(\leq \tilde{\lambda}_{k-1})$ has a $K$-basis consisting of bideterminants $T^{\tilde{\lambda}_k}(i:j)$, where
$T^{\tilde{\lambda}_k}_i, T^{\tilde{\lambda}_k}_j$ are standard tableaux.

The $K$-span of bideterminants $T^{\tilde{\lambda}_k}(i:\ell_{\tilde{\lambda}_k})$ is isomorphic to the left costandard $S(m,\tilde{r})$-module $\nabla_{S(m,\tilde{r})}(\tilde{\lambda}_k)$, 
the $K$-span of bideterminants $T^{\tilde{\lambda}_k}(\ell_{\tilde{\lambda}_k}:j)$ is isomorphic to the right $S(m,\tilde{r})$-module $\Delta_{S(m,r)}(\tilde{\lambda}_k)^*$ that is dual to the right standard module 
$\Delta_{S(m,\tilde{r})}(\tilde{\lambda}_k)$,
and the quotient $A(\leq \tilde{\lambda}_k)/A(\leq \tilde{\lambda}_{k-1})$, represented by the $K$-span of bideterminants $T^{\tilde{\lambda}_k}(i:j)$, is isomorphic to the $S(m,\tilde{r})$-bimodule $\nabla_{S(m,\tilde{r})}(\tilde{\lambda}_k)\otimes \Delta_{S(m,\tilde{r})}(\tilde{\lambda}_k)^*$
for each $k=1, \ldots, t_{\tilde{r}}$.

Now we turn our attention to $\GL(m)$-modules.
A dominant $\GL(m)$-weight $\lambda=(\lambda_1,\ldots, \lambda_m)\in X(T)^+$ is polynomial if $\lambda_m\geq 0$.
Denote $\nu=(1, \ldots, 1)$. 
To each $\lambda=(\lambda_1, \ldots, \lambda_m)\in X(T)^+$ we assign the corresponding partition 
$\tilde{\lambda}= \lambda-\lambda_m\nu= (\lambda_1-\lambda_m, \ldots, \lambda_m-\lambda_m=0)$.
If $|\lambda|=\sum_{i=1}^m \lambda_i=r$, then $|\tilde{\lambda}|=r-m\lambda_m=\tilde{r}$.
Partition $\tilde{\lambda}$ corresponds to a weight of $S(m,\tilde{r})$, which is a polynomial weight of $GL(m)$.

Modules over Schur algebras $S(m,r)$ correspond to polynomial $\GL(m)$-modules. 
If $\tilde{\lambda}$ is polynomial and $|\tilde{\lambda}|=\tilde{r}$, then the induced $G$-module $H^0_G(\tilde{\lambda})$ of the highest weight $\tilde{\lambda}$ is isomorphic to $\nabla_{S(m,\tilde{r})}(\tilde{\lambda})$.

The one-dimensional $\GL(m)$-module $Det$ of the weight $\nu$ plays a crucial role in the representation theory of $\GL(m)$.
If $\lambda\in X(T)^+$, then there is 
\[H^0_{\GL(m)}(\lambda)\simeq \nabla_{S(m,\tilde{r})}(\tilde{\lambda})\otimes_K Det^{\lambda_m},\]
and we say that 
$H^0_{\GL(m)}(\lambda)$ is obtained from $\nabla_{S(m,\tilde{r})}(\tilde{\lambda})$ using a shift by $Det^{\lambda_m}$.

Analogously,  
\[(V_{\GL(m)}(\lambda)^*)^r\simeq  \Delta_{S(m,\tilde{r})}(\tilde{\lambda})^*\otimes_K Det^{\lambda_m}.\]

Thus, if $\lambda$ is polynomial, then the $\GL(m)$-bimodule $H^0_{\GL(m)}(\lambda)\otimes (V_{\GL(m)}(\lambda)^*)^r$ is isomorphic to $\nabla_{S(m,r)}(\lambda)\otimes \Delta_{S(m,r)}(\lambda)^*$. 
If $\lambda$ is not polynomial, then 
\[H^0_{\GL(m)}(\lambda)\otimes (V_{\GL(m)}(\lambda)^*)^r\simeq 
(\nabla_{S(m,\tilde{r})}(\tilde{\lambda})\otimes Det^{\lambda_m})\otimes (\Delta_{S(m,\tilde{r})}(\tilde{\lambda})^*\otimes Det^{\lambda_m}).\]

\subsection{Building a specific Donkin-Koppinen order}
\label{ss33}

Theorem 2 of \cite{kop} states that Donkin-Koppinen filtrations are in one-to-one correspondence with linear orders of 
all dominant weights \{$\lambda_1, \ldots, \lambda_n, \ldots \}$ of $G$ that satisfies the following property:
\begin{equation}\label{2}
\begin{aligned}
&\text{If an irreducible } G-\text{module } L_G(\lambda) \text{ of the highest weight } \lambda \text{ appears as a} \\
&\text{composition factor of } H^0_G(\mu),  \text{ then } \lambda  \text{ appears in the above list before }  \mu.
\end{aligned}
\end{equation}

If a linear order of dominant weights satisfies $(\ref{2})$, then it is called a Donkin-Koppinen order. We will now construct a specific Donkin-Koppinen order.

Since all weights in the induced module $H^0_G(\lambda)$ have the same length $|\lambda|=r$, 
the order of weights of different lengths does not influence whether the property $(\ref{2})$ is satisfied. Therefore, it suffices to describe the order of dominant weights $\lambda$ of fixed length $r$. 

If $\lambda\in X(T)^+$, then the corresponding partition $\tilde{\lambda}$ appears in the listing $(\ref{1})$ of weights of length $\tilde{r}$.
To build a listing of dominant weights of length $r$, we proceed by sets corresponding to $\tilde{r}\geq 0$ such that $r\equiv \tilde{r} \pmod m$. We start with the smallest such $\tilde{r}\geq 0$ and proceed to larger $\tilde{r}$, in each step adding $m$. For each $\tilde{r}$, we order all partitions $\tilde{\lambda}$ of length $\tilde{r}$ such that $\tilde{\lambda}_m=0$ as in $(\ref{1})$, and then shift them to weights $\lambda$ by subtracting $\frac{\tilde{r}-r}{m}\nu$. We list these weights $\lambda$ in the corresponding  order by increasing value of $\tilde{r}$, and obtain an infinite order of all dominant weights of length $r$.
Note that a weight $\lambda$ corresponding to a shift of $\tilde{\lambda}$ such that 
$\tilde{\lambda}_m>0$ appears in the previous part of the list corresponding to the value $\tilde{r}-\tilde{\lambda}_m m$.

This listing of dominant weights of length $r$ behaves like "a nested Russian doll chain," in which the list of the weights corresponding to values $0\leq \tilde{r}\leq \tilde{r}_0$ appears at the beginning of the list of the weights corresponding to values $0\leq \tilde{r}\leq \tilde{r}_1$ for $\tilde{r}_0<\tilde{r}_1$  as seen on the following examples.

\begin{ex}
Assume $m=3$ and $r=0$.
Consider the listing of weights $\tilde{\lambda}$ 

$(0,0,0)$ for $\tilde{r}=0$;

$(1,1,1)<(2,1,0)<(3,0,0)$ for $\tilde{r}=3$;

$(2,2,2)<(3,2,1)<(3,3,0)<(4,2,0)<(5,1,0)<(6,0,0)$ for $\tilde{r}=6$;

$(3,3,3)<(4,3,2)<(4,4,1)<(5,4,0)<(6,3,0)<(7,2,0)<(8,1,0)<(9,0,0)$ for $\tilde{r}=9$.

Corresponding to these, we obtain the beginning of the order of dominant weights $\lambda$ of length $r=0$
as
\[\begin{aligned} &(0,0,0);\\
&<(1,0,-1)<(2,-1,-1);\\
&<(1,1,-2)<(2,0,-2)<(3,-1,-2)<(4,-2,-2);\\
&<(2,1,-3)<(3,0,-3)<(4,-1,-3)<(5,-2,-3)<(6,-3,-3).
\end{aligned}\]
For $m=3$ and $r=-4$, considering the listings for $\tilde{r}=2,5,8$, we get the listing of $\lambda$ as
\[\begin{aligned} &(-1,-1,-2)<(0,-2,-2);\\
&<(0,-1,-3)<(1,-2,-3)<(2,-3,-3);\\
&<(0,0,-4)<(1,-1,-4)<(2,-2,-4)<(3,-3,-4)<(4,-4,-4). \end{aligned}\]
For $m=3$ and $r=4$, considering the listings for $\tilde{r}=4, 7$, we get the listing of $\lambda$ as
\[\begin{aligned}&(2,1,1)<(2,2,0)<(3,1,0)<(4,0,0);\\
&<(3,2,-1)<(4,1,-1)<(5,0,-1)<(6,-1,-1).\end{aligned}\]
\end{ex}

The specific Donkin-Koppinen order of dominant $\GL(m)$-weights constructed above will be denoted by $\leq$.
It follows from Lemma \ref{l12.1} that $\leq$ refines the dominance order $\unlhd$ of dominant $\GL(m)$-weights.

\subsection{Donkin-Koppinen filtration for $\GL(m)$}

Let $\lambda$ be a dominant $\GL(m)$-weight and $\tilde{\lambda}=\lambda-\lambda_m\nu$ be the corresponding polynomial weight as above.
Fix a basic tableau $T^{\tilde{\lambda}}$ of the shape $\tilde{\lambda}$ and the length $\tilde{r}=|\tilde{\lambda}|$. Let $i,j$ be multi-indices of length $\tilde{r}$ with entries from the set $\{1, \ldots, m\}$ and $T^{\tilde{\lambda}}(i:j)$ be the bideterminant built on $i,j$.

Denote by $M_{\leq \tilde{\lambda}}$ the $K$-span of all bideterminants of shapes $\tilde{\mu}\leq \tilde{\lambda}$, and by 
$M_{\unlhd \tilde{\lambda}}$ the $K$-span of all bideterminants of shapes $\tilde{\mu}\unlhd \tilde{\lambda}$. 
Also, denote by $M_{< \tilde{\lambda}}$ the $K$-span of all bideterminants of shapes $\tilde{\mu}< \tilde{\lambda}$, and by 
$M_{\lhd \tilde{\lambda}}$ the $K$-span of all bideterminants of shapes $\tilde{\mu}\lhd \tilde{\lambda}$. 
It is a well-known result (see \cite{martin}) that $M_{\unlhd \tilde{\lambda}}$ is a $\GL(m)$-bimodule that has a $K$-basis consisting of 
bideterminants $T^{\tilde{\zeta}}(i:j)$, where $\tilde{\zeta}\unlhd \tilde{\lambda}$ and $T^{\tilde{\zeta}}_i, T^{\tilde{\zeta}}_j$ are standard tableaux.
Also, $M_{\lhd \tilde{\lambda}}$ is a $\GL(m)$-bimodule that has a $K$-basis consisting of 
bideterminants $T^{\tilde{\zeta}}(i:j)$, where $\tilde{\zeta}\lhd \tilde{\lambda}$ and $T^{\tilde{\zeta}}_i, T^{\tilde{\zeta}}_j$ are standard tableaux.
Additionally, $M_{\unlhd \tilde{\lambda}}/M_{\lhd \tilde{\lambda}} \simeq H^0_{\GL(m)}(\lambda)\otimes (V_{\GL(m)}(\lambda)^*)^r$ 
as $\GL(m)$-bimodules.

Shifting by a power of $Det$, we define $M_{\leq \lambda}= Det^{\lambda_m}\otimes M_{\leq \tilde{\lambda}}$ and 
$M_{\unlhd \lambda}= Det^{\lambda_m}\otimes M_{\unlhd \tilde{\lambda}}$. Then 
$M_{\unlhd \lambda}$ is a $\GL(m)$-bimodule that has a $K$-basis consisting of 
$Det^{\lambda_m}\otimes T^{\tilde{\zeta}}(i:j)$, where $\tilde{\zeta}\unlhd \tilde{\lambda}$ and $T^{\tilde{\zeta}}_i, T^{\tilde{\zeta}}_j$ are standard tableaux. 
The expressions $Det^{\lambda_m}\otimes T^{\tilde{\zeta}}(i:j)$ are called \emph{generalized bideterminants}.
Also, 
$M_{\lhd \lambda}$ is a $\GL(m)$-bimodule that has a basis consisting of 
$Det^{\lambda_m}\otimes T^{\tilde{\zeta}}(i:j)$, where $\tilde{\zeta}\lhd \tilde{\lambda}$, and $T^{\tilde{\zeta}}_i, T^{\tilde{\zeta}}_j$ are standard tableaux. 
Finally, 
\[M_{\unlhd \lambda}/M_{\lhd \lambda} \simeq H^0_{\GL(m)}(\lambda)\otimes (V_{\GL(m)}(\lambda)^*)^r\]
as $\GL(m)$-bimodules.

If $\Gamma$ is a finitely generated ideal of $X(T)^+$, then the $\GL(m)$-bimodules $M_{\leq \lambda}$ for $\lambda\in \Gamma$ form a Donkin-Koppinen filtration of $O_{\Gamma}(K[\GL(m)])$ since 
\[M_{\leq \lambda}/M_{< \lambda} \simeq H^0_{\GL(m)}(\lambda)\otimes (V_{\GL(m)}(\lambda)^*)^r\]
as $\GL(m)$-bimodules.

\subsection{A basis of $M_{\Gamma}$}

If $G=\GL(m|n)$ is a general linear supergroup, then its even subgroup $G_{ev}$ is isomorphic to $\GL(m)\times \GL(n)$.
We write a dominant weight $\lambda$ of $G$ as $\lambda=(\lambda^+|\lambda^-)$, where 
$\lambda^+$ is a dominant weight of $\GL(m)$ of length $r^+=|\lambda^+|$, and $\lambda^-$ is a dominant weight of $\GL(n)$ of length $r^-=|\lambda^-|$.
We define the strong dominance order $\unlhd_s$ on dominant weights of $\GL(m|n)$ as
$\mu\unlhd_s \lambda$ if and only if $\mu^+\unlhd \lambda^+$ and $\mu^-\unlhd \lambda^-$.

If a simple $G_{ev}$-module of highest weight $\mu$ is a composition factor of $H^0_{G_{ev}}(\lambda)$, 
then $\mu\unlhd_s \lambda$. We can combine the previously defined orders $\leq$ on the weights $\lambda^+$ of length $r^+$, and $\lambda^-$ of length $r^-$ of  $\GL(m)$,  and $\GL(n)$ respectively, 
to create a lexicographic order $\leq_{lex}$ on weights of $G$ by
$\mu\leq_{lex} \lambda$ if and only if $\mu^+<\lambda^+$, or $\mu^+=\lambda^+$ and $\mu^-\leq \lambda^-$.
Then $\leq_{lex}$ is a Donkin-Koppinen order for $G_{ev}$.

We can use the specific listings of dominant weights $\{\lambda^+_1, \ldots, \lambda^+_i, \ldots \}$ of $\GL(m)$ of length $r^+$, and 
$\{\lambda^-_1, \ldots, \lambda^-_i, \ldots, \}$ of $\GL(n)$ of length $r^-$ constructed earlier, and combine them by induction on $i$ to get a listing of dominant weights $\lambda=(\lambda^+|\lambda^-) $ 
of $\GL(m|n)$ of length $(r^+|r^-)$ as follows.
For $i=1$, we list the minimal weight $(\lambda^+_1| \lambda^-_1)$. For the inductive step, we assume that all weights 
$(\lambda^+_j|\lambda^-_k)$, where $1\leq j,k\leq i$, were listed, and add the remaining weights of the type
$(\lambda^+_j|\lambda^-_k)$, where $1\leq j,j\leq i+1$ in the order 
\[(\lambda^+_{i+1}|\lambda^-_1), \ldots (\lambda^+_{i+1}|\lambda^-_i), (\lambda^+_1|\lambda^-_{i+1}), 
\ldots, (\lambda^+_i|\lambda^-_{i+1}), (\lambda^+_{i+1}|\lambda^-_{i+1}).\]
Denote the above listing of dominant weights $\lambda$ of $\GL(m|n)$ of fixed length $(r^+|r^-)$ 
as $\{\mu^{r^+|r^-}_1, \ldots, \mu^{r^+|r^-}_i, \ldots \}$ and the corresponding order by $\leq^{r^+|r^-}$.

Since there are no extensions between simple $G_{ev}$-modules of different lengths $(r^+|r^-)$, we can build 
Donkin-Koppinen order for $G_{ev}$ in many ways. 
One possibility is to choose an arbitrary listing of the ordered pairs $(r^+|r^-)\in \mathbb{Z}^2$, and then form the listing of all weights $\lambda$ by splicing the listings corresponding to $\leq^{r^+|r^-}$ in the chosen order of ordered pairs $(r^+|r^-)$.

We can construct yet another Donkin-Koppinen order $\leq_{ev}$ of dominant weights $\lambda$ of $G_{ev}$ as follows. 
First, we list the minimal weights $\mu^{r^+|r^-}_1$ for all ordered pairs $(r^+|r^-)$ in an arbitrary order.
Then we list the weights $\mu^{r^+|r^-}_2$ for all ordered pairs $(r^+|r^-)$ in an arbitrary order, and so on, 
each time listing weights $\mu^{r^+|r^-}_i$ for a fixed index $i$ and all ordered pairs $(r^+|r^-)$ in an arbitrary order. 

Any of the orders $\leq$, constructed above, refines the order $\unlhd_s$.

The $K$-space $M_{\leq\lambda}= \sum_{\mu\leq \lambda} M_{\leq \mu^+}M_{\leq \mu^-}$ is a $G_{ev}$-bimodule 
and it has a $K$-basis given by vectors $v=v^+v^-$, where 
\[v^+=Det_1^{\lambda^+_m} T^{\tilde{\zeta}^+}(i^+:j^+)\text{ and } v^-=Det_2^{\lambda^-_n}T^{\tilde{\zeta}^-}(i^-:j^-)\]
are generalized bideterminants such that $(\tilde{\zeta}^+|\tilde{\zeta}^-)\leq \tilde{\lambda}$,  and tableaux $T^{\tilde{\zeta}^+}_{i^+}, T^{\tilde{\zeta}^+}_{j^+}$ and $T^{\tilde{\zeta}^-}_{i^-}, T^{\tilde{\zeta}^-}_{j^-}$ are standard.
If $\mu$ is a predecessor of $\lambda$ under $\leq$, then we define $M_{<\lambda}=M_{\mu}$.

The advantage of using the order $\leq_{ev}$  is that each bimodule $M_{\leq_{ev} \lambda}$ is a direct sum of finite-dimensional bimodules $M_{\leq_{ev} \lambda}^{r^+|r^-}$ consisting of all elements in $M_{\leq_{ev} \lambda}$ of the length $(|\lambda^+|=r^+|r^-=|\lambda^-|)$. 

If $\Gamma$ is a finitely generated ideal of $X(T)^+$, then we can describe the Donkin-Koppinen filtration of
$M_{\Gamma}$ as follows.

\begin{pr}
Let $\Gamma$ be a finitely generated ideal of $X(T)^+$. Then the set of $G_{ev}$-bimodules $M_{\leq \lambda}$, 
listed by the restriction of the Donkin-Koppinen order $\leq$ on $\Gamma$, form a Donkin-Koppinen filtration of 
$M_{\Gamma}=O_{\Gamma}(K[G])$.
\end{pr} 
\begin{proof}
Denote $M_{\unlhd_s \lambda}= M_{\unlhd \lambda^+}M_{\unlhd \lambda^-}$
and $M_{\lhd_s \lambda}= M_{\unlhd \lambda^+} M_{\lhd \lambda^-}
+M_{\lhd \lambda^+}M_{\unlhd \lambda^-}$.

The isomorphisms
\[M_{\leq \lambda}/M_{<\lambda}\simeq H^0_{G_{ev}}(\lambda)\otimes (V_{G_{ev}}(\lambda)^*)^r\simeq M_{\unlhd_s \lambda}/M_{\lhd_s \lambda}\]
of $G_{ev}$-bimodules proves the claim.
\end{proof}

\section{Donkin-Koppinen filtration for $\GL(m|n)$}

\subsection{The action of superderivations on generators}

Corresponding to the block decomposition
\[\begin{pmatrix}C_{11}&C_{12}\\C_{21}&C_{22}\end{pmatrix}\]
of the generic $(m+n)\times(m+n)$ matrix $C$, we break the ordered pairs $(i,j)$ of indices from the set $\{1, \ldots, m+n\}$ into four blocks. 
We write 

$(i,j)\in I_{11}$ if and only of $1\leq i,j\leq m$;

$(i,j)\in I_{12}$ if and only if $1\leq i\leq m$ and $m+1\leq j\leq m+n$;

$(i,j)\in I_{21}$ if and only if $m+1\leq i\leq m+n$ and $1\leq j\leq m$;

$(i,j)\in I_{22}$ if and only if $m+1\leq i,j\leq m+n$.

We have the following formulae:

If $(i,j)\in I_{11}$, then $y_{ij}=c_{ij}$.

If $(i,j)\in I_{12}$, then \[y_{ij}=\frac{A_{i1}c_{1j}+\ldots +A_{im}c_{mj}}{D}.\]

If $(i,j)\in I_{21}$, then \[y_{ij}=\frac{c_{i1} A_{1j}+\ldots +c_{im}A_{mj}}{D}.\]

If $(i,j)\in I_{22}$, then \[y_{ij}=c_{ij}-c_{i1}y_{1j}-\ldots - c_{im}y_{mj}
=c_{ij}-\frac{\sum_{u,v=1}^m c_{iu}A_{uv}c_{vj}}{D}.\]

Additionally, $(D)_{kl}D=0$ for $(kl)\in I_{11}\cup I_{22}$, and 
$(D)_{kl}D=D y_{kl}$ for $(kl)\in I_{12}\cup I_{21}$.

\begin{lm}\label{l5.1}
If $(ij)\in I_{21}$ and $(kl)\in I_{12}$, then $(y_{ij})_{kl}D=\frac{A_{kj} y_{il}}{D}$.
\end{lm}
\begin{proof}
We compute 
$(y_{ij})_{kl}D= (\frac{c_{i1}A_{1j}+\ldots + c_{im}A_{mj}}{D})_{kl}D=$
\[\frac{D(c_{i1}A_{1j}+\ldots + c_{im}A_{mj})_{kl}D-
(c_{i1}A_{1j}+\ldots + c_{im}A_{mj})(A_{k1}c_{1l}+\ldots+A_{km}c_{ml})}{D^2}.\]
Using the Laplace expansion 
\[A_{uj}=\sum_{v\neq j} (-1)^{u+k+v+j} A(uk|vj) c_{vk},\] 
we rewrite
\[(c_{i1}A_{1j}+\ldots + c_{im}A_{mj})_{kl}D
=c_{il}A_{kj}+\sum_{u\neq k}\sum_{v\neq j} (-1)^{u+k+v+j}c_{iu} A(uk|vj) c_{vl}.
\]
Break the expression
\[D(c_{i1}A_{1j}+\ldots + c_{im}A_{mj})_{kl}D-
(c_{i1}A_{1j}+\ldots + c_{im}A_{mj})(A_{k1}c_{1l}+\ldots+A_{km}c_{ml})\]
into two parts. The first part consists of all terms 
\[\begin{aligned}&
c_{il}A_{kj}D- \sum_{v=1}^m c_{ik}A_{kj}A_{kv}c_{vl}-\sum_{u=1}^m c_{iu}A_{uj} A_{kj}c_{jl}-c_{ik}A_{kj}^2c_{jl}\\
&=A_{kj}(c_{il}D-\sum_{v=1}^m c_{ik}A_{kv}c_{vl}-\sum_{u=1}^m c_{iu}A_{uj}c_{jl}-c_{ik}A_{kj}c_{jl})
\end{aligned}\]
that are multiples of $A_{kj}$.

The second part consists of terms that are multiples of $c_{iu}c_{vl}$, where $u\neq k$ and $v\neq j$.
The multiples of $c_{iu}c_{vl}$, where $u\neq k$ and $v\neq j$, are
\[\begin{aligned}&-c_{iu}A_{uj}A_{kv}c_{vl}+D(-1)^{u+k+v+j}c_{iu}A(uk|vj)c_{vl}\\
&=c_{iu}(-A_{uj}A_{kv}+D(-1)^{u+k+v+j}A(uk|vj))c_{vl}\\
&=-c_{iu}A_{uv}A_{kj}c_{vl}
\end{aligned}\]
using the Jacobi theorem on minors of the adjoint matrix stating that
\[A_{uj}A_{kv}-A_{uv}A_{kj}=D(-1)^{u+j+k+v}A(uk|vj).
\]
Combining the expression for all terms we obtain 
\[\begin{aligned}&D(c_{i1}A_{1j}+\ldots + c_{im}A_{mj})_{kl}D-
(c_{i1}A_{1j}D+\ldots + c_{im}A_{mj})(A_{k1}c_{1l}+\ldots+A_{km}c_{ml})\\
&=A_{kj}(c_{il}D-\sum_{u,v=1}^m c_{iu}A_{uv}c_{vl})=DA_{kj}y_{il}.
\end{aligned}\]
Thus $(y_{ij})_{kl}D=\frac{A_{kj}y_{il}}{D}.$
\end{proof}

The action of all right superderivations $_{kl}D$ on elements $y_{ij}$ is described in the following table.
This table is of independent interest, and completes the formulae established in \cite{irred} and \cite{prim}.

\begin{pr}\label{p12.2}
The values $(y_{ij})_{kl}D$ are given as follows.
\[\begin{array}{l|cccccc}
(ij)\backslash (kl) & &  I_{11}  &  \qquad I_{12} &  I_{21} & \quad I_{22} \\
\hline\\
I_{11}&&\delta_{jk} y_{il}&\quad \delta_{jk}\sum_{u=1}^m y_{iu}y_{ul}&0&\quad 0\\
I_{12}&&-(1-\delta_{ki})\delta_{il}y_{kj}&\quad y_{il}y_{kj}&\delta_{jk}\delta_{il}&\quad \delta_{jk}y_{il}\\
I_{21}&&0&\quad \frac{y_{il}A_{kj}}{D}&\quad 0&\quad 0\\
I_{22}&&0&\quad y_{il}y_{kj}&\quad 0&\quad \delta_{jk} y_{il}
\end{array}\]
\end{pr}
\begin{proof}
If $(ij), (kl)\in I_{11}$, then $(c_{ij}) _{kl}D = \delta_{jk} c_{il}$ from the definition of $_{kl}D$.

If $(ij)\in I_{11}$ and $(kl)\in I_{12}$, then 
$(c_{ij})_{kl}D= \delta_{jk} c_{il}$. Since $c_{il}=\sum_{u=1}^m c_{uj}y_{ul}$, the formula follows.

The cases $(ij)\in I_{11}$ and $(kl)\in I_{21}\cup I_{22}$ are trivial.

The case $(ij) \in I_{12}$ and $(kl)\in I_{11}$ follows from Lemma 6.1 of \cite{irred}.

The case $(ij), (k,l) \in I_{12}$ follows from Lemma 2.1 of \cite{irred}.

If $(i,j)\in I_{12}$, $(kl)\in I_{21}$, then $(y_{ij})_{kl}D=\delta_{jk} \frac{A_{i1}c_{il}+ \ldots + A_{im}c_{ml}}{D}$. Since $A_{i1}c_{il}+ \ldots + A_{im}c_{ml}=\delta_{il}D$, the formula follows.

The case $(ij)\in I_{12}$ and $(kl)\in I_{22}$ follows from Lemma 6.1 of \cite{irred}.

If $(ij)\in I_{21}$ and $(kl)\in _{11}$, then using the arguments from the proof of Lemma 6.1 of \cite{irred} we compute
$(y_{ij})_{kl}D=c_{il}A_{kj}-c_{il}A_{kj}=0$.

The case $(ij)\in I_{21}$ and $(kl)\in I_{12}$ follows from Lemma \ref{l5.1}.

The cases $(ij)\in I_{22}$ and $(kl)\in I_{21}\cup I_{22}$ are trivial.

If $(ij)\in I_{22}$ and $(kl)\in I_{11}$, then using Lemma 6.1 of \cite{irred} we compute
$(y_{ij})_{kl}D=-c_{il}y_{kj}-(-c_{il}y_{kj})=0$.

If $(ij)\in I_{22}$ and $(kl)\in I_{12}$, then using Lemma 2.1 of \cite{irred} we compute
$(y_{ij})_{kl}D=c_{il}y_{kj}-c_{i1}y_{1l}y_{kj} - \ldots - c_{im}y_{ml}y_{kj}=y_{il}y_{kj}$.

If $(ij)\in I_{22}, (kl)\in I_{21}$, then we use the formula from the case $(ij) \in I_{12}, (kl)\in I_{21}$ to derive
$(y_{ij})_{kl}D=(1-\delta_{jk})(c_{il}-c_{il})=0$.

If $(ij)\in I_{22}$ and $(kl)\in I_{22}$, then using Lemma 6.1 of \cite{irred} we compute
$(y_{ij})_{kl}D=\delta{jk}(c_{il} - c_{i1}y_{1l}-\ldots - c_{im}y_{ml})=\delta_{jk}y_{il}$.
\end{proof}

Analogously, we obtain the following formulae.

\begin{pr}\label{p12.3}
The values $D_{kl}(y_{ij})$ are given as follows.
\[\begin{array}{l|cccccc}
(ij)\backslash (kl) & &  I_{11}  &  \qquad I_{21} &  I_{12} & \quad I_{22} \\
\hline\\
I_{11}&&\delta_{il} y_{kj}&\quad \delta_{il}\sum_{u=1}^m y_{ku}y_{uj}&0&\quad 0\\
I_{21}&&-(1-\delta_{lj})\delta_{kj}y_{il}&\quad y_{kj}y_{il}&\delta_{il}\delta_{jk}&\quad \delta_{il}y_{kj}\\
I_{12}&&0&\quad \frac{y_{kj}A_{il}}{D}&\quad 0&\quad 0\\
I_{22}&&0&\quad y_{kj}y_{il}&\quad 0&\quad \delta_{il} y_{kj}
\end{array}\]
\end{pr}
\begin{proof}
Interchange $2\leftrightarrow 1$, $i\leftrightarrow j$, $l\leftrightarrow k$ and $y_{uv}\rightarrow y_{vu}$
in the table of Lemma \ref{l5.1}.
\end{proof} 

\begin{rem}
All divided powers $_{kl}D^{(e)}$ and $D_{kl}^{(e)}$, where $1\leq k\neq l\leq m$ or $m+1\leq k\neq l\leq m+n$ and $e>1$ vanish on all $y_{ij}$. 
This follows from $(y_{ij})_{kl}D^2 =0$ unless $k=l=j$, and $D^2_{kl}(y_{ij})=0$ unless $k=l=i$.
\end{rem}

\subsection{$G$-bimodules $C_{\leq \lambda}$}

\begin{lm}\label{l6.4}
$\phi^*(M_{\leq \lambda})$ is a $G$-superbimodule of $\phi^*(K[Y_{21}]\otimes M_{\leq \lambda}\otimes K[Y_{12}])$.
\end{lm}
\begin{proof}
We have already mentioned that $M_{\leq \lambda}$ is a $G_{ev}$-bimodule. It was observed in \cite{prim} that $\phi^*$ is a $G_{ev}$-morphism. Therefore $\phi^*(M_{\leq \lambda})$
is a $G_{ev}$-bimodule of $\phi^*(K[Y_{21}]\otimes M_{\leq \lambda}\otimes K[Y_{12}])$.

Next, consider the action of odd superderivation $_{kl}D$ for $(kl)\in I_{12}$. Let $v^+=D^{b^+}T^{\zeta^+}(i^+:j^+)$ be a generalized bideterminant.
Denote by $v_u$, for $u=1, \ldots, m$, 
the sum of generalized bideterminants obtained by replacing one of the entries $i$ in the tableau  $T^{\zeta^+}_j$ by the letter $u$. By Proposition \ref{p12.2}, $(\phi^*(v^+))_{kl}D$ is a sum  $\sum_{u=1}^m \phi^*(v_u) y_{ul}$. Using the straightening algorithm for bideterminants,
we can replace all appearing generalized bideterminants by linear combinations of basis elements of 
$M_{\unlhd \zeta^+}$. Applying $\phi^*$ to those elements, we infer that 
$(\phi^*(v^+))_{kl}D\in \phi^*(K[Y_{21}]\otimes M_{\leq \lambda}\otimes K[Y_{12}])$.
Using Proposition \ref{p12.3}, we obtain an analogous statement for $\phi^*(v^+)$ and $D_{kl}$, where $(kl)\in I_{21}$.

Let $v^-=D^{b^-}T^{\zeta^-}(i^-:j^-)$ be a generalized bideterminant. Denote by $w$ the sum of generalized bideterminants obtained by replacing one of the entries $i$ in 
$T^{\zeta^-}_{i^-}$ by the letter $l$. By Proposition \ref{p12.2}, $(\phi^*(v^-))_{kl}D=\phi^*(w) y_{kj}$. Using the straightening algorithm for bideterminants appearing in $w$, we can replace them by linear combinations of basis elements in $M_{\unlhd \zeta^-}$.
Applying $\phi^*$ to those elements, we infer that $(\phi^*(v^-))_{kl}D\in \phi^*(K[Y_{21}]\otimes M_{\leq \lambda}\otimes K[Y_{12}])$.
Using Proposition \ref{p12.3}, we obtain an analogous statement for $\phi^*(v^-)$ and $D_{kl}$, where $(kl)\in I_{12}$.
\end{proof}

We will denote $M^*_{\leq\lambda}= \phi^*(M_{\leq \lambda})$.

Assume that $|\lambda^+|=r^+$ and $|\lambda^-|=r^-$.
Corresponding to the action of the odd root $\alpha=\epsilon_m-\epsilon_{m+1}$, together with the $G_{ev}$-bimodule
$M_{\lambda,0}=M_{\leq \lambda}$, we consider also $M_{\lambda,l}=M_{\leq \lambda-l\alpha}$ for $l\geq 0$.
The $K$-space $M_{\lambda,l}$ is a finite-dimensional $G_{ev}$-bimodule, and the lengths $(|\mu^+|||\mu^-|)$ of its weights $\mu$ range from 
$(r^+-l|r^-+l)$ to $(r^+-l-mn|r^-+l+mn)$.

We define 
\[C_{\leq \lambda}=\phi^*(\oplus_{l\geq 0}K[Y_{21}]\otimes (\oplus_{l\geq 0} M_{\lambda,l}) \otimes K[Y_{12}]),\] and 
\[C_{< \lambda}= \phi^*(K[Y_{21}]\otimes (M_{<\lambda} \oplus \oplus_{l>0} M_{\lambda,l}) \otimes K[Y_{12}] ).\]

Recall from Example 5.1 of \cite{sz}, that every dominant weight $\lambda$ of $G$ has finitely many predecessors with respect to the order $\unlhd$. One predecessor is $\lambda-\alpha$, where $\alpha=\epsilon_{m}-\epsilon_{m+1}$; and the remaining predecessors are found among weights $\mu$ that satisfy $\mu\lhd_s \lambda$.

In the terminology of Section 3 of \cite{markozub}, the set of all weights of $G$ is interval-finite and good. 
By Proposition 3.10 of \cite{markozub}, every finitely generated ideal $\Gamma$ of the set of all weights of $G$  
has a descending chain of finitely-generated subideals
\[\ldots \subseteq \Gamma_{k+1}\subseteq \Gamma_k \subseteq \ldots \subseteq \Gamma_1\subseteq \Gamma_0=\Gamma\]
such that $\Gamma_k \setminus \Gamma_{k+1}$ is finite for every $k\geq 0$ and
$\cap_{k\geq 0} \Gamma_k = \emptyset$. We choose a unique such filtration for which the elements of $\Gamma_k \setminus \Gamma_{k+1}$ are pairwise incomparable generators of $\Gamma_k$ for each $k\geq 0$.

Assume an ideal $\Gamma$ is generated by incomparable weights $\lambda_1, \ldots, \lambda_s$.
Then $\Gamma_0\setminus \Gamma_1=\{\lambda_1, \ldots, \lambda_s\}$, and 
$\Gamma_1\setminus \Gamma_2=\{\mu_1, \ldots, \mu_t\}$
is a subset of all predecessors of weights $\lambda_1, \ldots, \lambda_s$, 
$\Gamma_2\setminus\Gamma_3$ is a subset of all predecessors of weights $\mu_1, \ldots, \mu_t$, and so on.

We denote by $C_{\leq \Gamma_k}$ the sum of $C_{\leq \mu}$, where $\mu\in \Gamma_k$.

\begin{tr}
Let $\Gamma$ be a finitely generated ideal of $X(T)^+$.
Corresponding to the descending chain of finitely-generated subideals of $\Gamma$
\[\ldots \subseteq \Gamma_{k+1}\subseteq \Gamma_k \subseteq \ldots \subseteq \Gamma_1\subseteq \Gamma_0=\Gamma\]
as above, 
there is a descending filtration of $C_{\Gamma}=O_{\Gamma}(K[G])$ by locally finite-dimensional $G$-superbimodules
\[\ldots \subseteq C_{\leq \Gamma_{k+1}}\subseteq C_{\leq \Gamma_k} \subseteq \ldots \subseteq C_{\leq \Gamma_1}\subseteq C_{\leq \Gamma_0}=\mathcal{O}_{\Gamma}(K[G])=C_{\Gamma},\]\
where $C_{\Gamma_k}=O_{\Gamma_k}(K[G])$ and
\[C_{\leq \Gamma_k}/C_{\leq \Gamma_{k+1}} \simeq \oplus_{\mu\in \Gamma_k\setminus \Gamma_{k+1}}
H^0_G(\mu)\otimes (V_G(\mu)^*)^r\]
as $G$-superbimodules.
If the chain of ideals 
\[\ldots \subseteq \Gamma_{k+1}\subseteq \Gamma_k \subseteq \ldots \subseteq \Gamma_1\subseteq \Gamma_0=\Gamma\]
is refined to a chain 
\[\ldots \subseteq \Gamma'_{l+1}\subseteq \Gamma'_l \subseteq \ldots \subseteq \Gamma'_1\subseteq \Gamma'_0=\Gamma,\]
such that each $\Gamma'_l\setminus \Gamma'_{l+1}$ consists of a single weight $\mu'_l$, then 
\[C_{\leq \Gamma'_l}/C_{\leq \Gamma'_{l+1}} \simeq  H^0_G(\mu'_l)\otimes (V_G(\mu'_l)^*)^r\]
as $G$-superbimodules, and the chain
\[\ldots \subseteq C_{\leq \Gamma'_{l+1}}\subseteq C_{\leq \Gamma'_l} \subseteq \ldots \subseteq C_{\leq \Gamma'_1}\subseteq C_{\leq \Gamma'_0}=\mathcal{O}_{\Gamma}(K[G])\]
is a Donkin-Koppinen filtration of $\mathcal{O}_{\Gamma} (K[G])=O_{\Lambda}(K[G])$.
\end{tr}
\begin{proof}
We prove the statement for the refined chain of ideals $\{\Gamma'_l\}$. The statement for the chain 
$\{\Gamma_k\}$ then follows from it.

Assume $\lambda=\mu'_l\in \Gamma'_l\setminus \Gamma'_{l+1}$. Then $(\lambda]\cap \Gamma'_{l+1}=(\lambda)\subseteq C_{\Gamma'_{l+1}}$.
It is enough to show that $C_{\leq \Gamma'_l}$ is a $G$-superbimodule and 
$C_{\leq \Gamma'_l}/C_{\leq \Gamma'_{l+1}}\simeq  H^0_G(\lambda)\otimes (V_G(\lambda)^*)^r$ as $G$-superbimodules.

We will show that $C_{\leq \lambda}$ is a left $G$-supermodule. Analogously, we can prove that it is a right $G$-supermodule.
It follows from Lemma \ref{l6.4} that $\phi^*(M_{\leq \lambda})\subset \phi^*(K[Y_{21}]\otimes M_{\lambda,l}\otimes K[Y_{12}])\subset C_{\leq \lambda}$.
Then Propositions \ref{p12.2} and \ref{p12.3} imply that
$\phi^*(K[Y_{21}]\otimes M_{\lambda,l}\otimes K[Y_{12}])$
is invariant under the left and right actions of $\Dist(G_{ev})$. 

For $v\in M_{\lambda,l}$, denote $v^*=\phi^*(v)$. 
Propositions \ref{p12.2} and \ref{p12.3} imply that if $(kl)\in I_{12}$, $v\in M_{\lambda,l}$ and $y=\prod_{t=1}^s y_{i_t,j_t}$ is such that each $(i_t,j_t)\in I_{12}$,  then 
\[(v^*y)_{kl}D, D_{kl}(v^*y)\in \phi^*(K[Y_{21}]\otimes M_{\lambda,l}\otimes K[Y_{12}])\subset C_{\leq \lambda}.\]
Analogously, if $(kl)\in I_{21}$, $v\in M_{\lambda,l}$ and $y=\prod_{t=1}^s y_{i_t,j_t}$ is such that each $(i_t,j_t)\in I_{21}$,  then 
\[(v^*y)_{kl}D, D_{kl}(v^*y)\in \phi^*(K[Y_{21}]\otimes M_{\lambda,l}\otimes K[Y_{12}])\subset C_{\leq \lambda}.\]

It remains to investigate the expressions 
$(v^*y_{ij})_{kl}D=-(v^*)_{kl}Dy_{ij}+v^*\frac{y_{il}A_{kj}}{D}$ 
for $v\in M_{\lambda,l}$,  $(ij)\in I_{21}$ and $(kl)\in I_{12}$;
and $D_{kl}(v^*y_{ij})=D_{kl}(v^*)y_{ij}+v^*\frac{y_{kj}A_{il}}{D}$ for $v\in M_{\lambda,l}$, $(ij)\in I_{12}$ and $(kl)\in I_{21}$.

A closer look at the product $v\frac{c_{il}A_{kj}}{D}$ for $(ij)\in I_{21}$ and $(kl)\in I_{12}$, where 
$v\in M_{\lambda,l}$ is a generalized bideterminant of the shape $\zeta-l\alpha$, and $\zeta\unlhd_{ev} \lambda$ reveals that it can be obtained in the following way.
Write $v$ in the form $v=Det_1^{b} T^{\eta^+}(i^+:j^+)Det_2T^{\eta^-}(i^-:j^-)$, where $\eta^+_m=1$. In this case, the first column of $T^{\eta^+}_{i^+}$ 
and $T^{\eta^+}_{j^+}$ consists of entries $1, \ldots, m$ written in the increasing order down the first column.

We change the ordered pair of partitions $(\eta^+,\eta^-)$ into $(\kappa^+,\kappa^-)$ by removing the last square from the diagram $[\eta^+]$ and adding a new square at the end of the first row of $[\eta^-]$.
That means $\kappa^+_i=\eta^+_i$ for $i=1, \ldots, m-1$, $\kappa_m=0$;
$\kappa^-_1=\eta^-_1+1$ and $\kappa^-_j=\eta^-_j$.
Define $T^{\kappa^+}_{\overline{i}^+}$ by the restriction of $T^{\eta^+}_{i^+}$, 
$T^{\kappa^+}_{\overline{j}^+}$ by the restriction of $T^{\eta^+}_{j^+}$,
$T^{\kappa^-}_{\overline{i}^-}$ by extending $T^{\eta^-}_{i^-}$ by adding an entry $k$ at the end of the first row, 
and $T^{\kappa^-}_{\overline{j}^-}$ by extending $T^{\eta^-}_{j^-}$ by adding an entry $j$ at the end of the first row.
Denote $w=Det_1^{b^+}T^{\kappa^+}(\overline{i}^+:\overline{j}^+)Det_2T^{\kappa^-}(\overline{i}^-:\overline{j}^-)$.
Then $w^*=v^*\frac{y_{il}A_{kj}}{D}$ is a generalized bideterminant of shape $\zeta-(l+1)\alpha$, which belongs 
to  $\phi^*(K[Y_{21}]\otimes M_{\lambda,l+1}\otimes K[Y_{12}])\subset C_{\leq \lambda}$.

The situation is analogous for $v^*\frac{y_{kj}A_{il}}{D}$ for $v\in M_{\lambda,l}$, $(ij)\in I_{12}$, and $(kl)\in I_{21}$.
This proves that $C_{\leq \lambda}$ is a $G$-superbimodule. Therefore, $C_{\Gamma'_l}$ is a superbimodule for every $l\geq 0$.
Analogously, $C_{<\lambda}$ is a $G$-superbimodule as well. 

For a dominant $G$-weight $\lambda$, denote 
\[\gamma=(\gamma^+|\gamma^-)=(\lambda_1-\lambda_m,\ldots, \lambda_m-\lambda_m=0|\lambda_{m+1}-\lambda_{m+n}, \ldots, \lambda_{m+n}-\lambda_{m+n}=0)\] and by $T^{\gamma^+}_{\ell_{\gamma^+}}$ and 
$T^{\gamma^-}_{\ell_{\gamma^-}}$
the canonical tableau of shape $\gamma^+$ and $\gamma^-$, respectively.
Further, denote by $v$ the vector 
\[v=Det_1^{\lambda_m}Det_2^{\lambda_{m+n}}T^{\gamma^+}(\ell_{\gamma^+}:\ell_{\gamma^+})
T^{\gamma^-}(\ell_{\gamma^-}:\ell_{\gamma^-}).\]
Then the congruence class of $v^*=\phi^*(v)$ generates $C_{\leq \lambda}/C_{<\lambda}$ as a $G$-superbimodule, 
and $C_{\leq \lambda}/C_{<\lambda}$ is given by congruence classes of vectors
\[\prod_{1\leq i \leq m <j \leq m+n} y_{ij}^{\epsilon_{ij}}\phi^*(Det_1^{\lambda_m}Det_2^{\lambda_{m+n}}T^{\gamma^+}(i^+:j^+)T^{\gamma^-}(i^-:j^-)) \prod_{1\leq j \leq m <i \leq m+n} y_{ij}^{\epsilon_{ij}},\]
each each $\epsilon_{ij}\in \{0,1\}$.
Using the straightening algorithm, it follows that the $K$-basis of $C_{\leq \lambda}/C_{<\lambda}$ consists of congruence classes of
vectors 
\[\prod_{1\leq i \leq m <j \leq m+n} y_{ij}^{\epsilon_{ij}} \phi^*(Det_1^{\lambda_m}Det_2^{\lambda_{m+n}}T^{\gamma^+}(i^+:j^+)T^{\gamma^-}(i^-:j^-)) \prod_{1\leq j \leq m <i \leq m+n} y_{ij}^{\epsilon_{ij}},\]
where each $\epsilon_{ij}\in \{0,1\}$ and $T^{\gamma^+}_{i^+,}T^{\gamma^+}_{j^+}$, $T^{\gamma^-}_{i^-}$, $T^{\gamma^-}_{j^-}$ 
 are semistandard tableaux.

Finally, $C_{\leq \lambda}/C_{<\lambda} \simeq H^0_G(\lambda)\otimes (V_G(\lambda)^*)^r$ as $G$-superbimodules.

Since $C_{\Gamma'_l}/C_{\Gamma'_{l+1}}\simeq C_{\leq \lambda}/C_{< \lambda}$ as $G$-superbimodules, the claim follows.
\end{proof}
 
\begin{rem}The above proof also shows that a $G_{ev}$-bimodule $\phi^*(K[Y_{21}]\otimes M_{\lambda,l}\otimes K[Y_{12}])$ is not invariant under the action of superderivations $_{kl}D$ for $(kl)\in I_{12}$ and $D_{kl}$ for $(kl)\in I_{21}$.
\end{rem}

\end{document}